\documentclass[11pt,a4paper]{article}
\usepackage[T1]{fontenc}
\usepackage{lmodern}
\usepackage[margin=25mm]{geometry}
\usepackage{amsmath,amssymb,amsthm,listings,booktabs,array}
\usepackage[expansion=false]{microtype}
\usepackage[hidelinks]{hyperref}
\hypersetup{pdftitle={A new lower bound for the sixth Schur number: S(6) >= 537},pdfauthor={Zhiyang Chen, Ligeng Zhu, Fan Long}}
\newtheorem{theorem}{Theorem}
\newtheorem{proposition}[theorem]{Proposition}
\newtheorem{corollary}[theorem]{Corollary}
\theoremstyle{definition}
\newtheorem{definition}[theorem]{Definition}

\newcommand{\fold}{\operatorname{f}}

\title{A new lower bound for the sixth Schur number: $S(6)\ge537$}
\author{
  Zhiyang Chen\\[3pt]
  {\small University of Toronto}\\
  {\small\texttt{zhiychen@cs.toronto.edu}}
  \and
  Ligeng Zhu\\[3pt]
  {\small NVIDIA Research}\\
  {\small\texttt{ligengz@nvidia.com}}
  \and
  Fan Long\\[3pt]
  {\small University of Toronto}\\
  {\small\texttt{fanl@cs.toronto.edu}}
}
\date{22 October 2026}

\begin{document}
\maketitle

\begin{abstract}
We give an explicit partition of $\{1,\ldots,537\}$ into six sum-free sets,
establishing $S(6)\ge537$. This improves the lower bound 536 of Fredricksen and
Sweet (2000). Known recursive constructions then give the improved lower bounds
$S(8)\ge5\,372$, $S(11)\ge204\,208$, $S(13)\ge2\,042\,082$, and
$R_6(3)\ge539$. We also report 131 pairwise
inequivalent partitions found by the search. We describe their structure and
provide the complete partitions in the supplementary material. These partitions
provide starting points for future searches and structural studies of Schur partitions.
\end{abstract}

\section{Introduction}
A set $C$ of integers is \emph{sum-free} if there are no $x,y\in C$, not necessarily
distinct, with $x+y\in C$. The Schur number $S(k)$ is the largest $n$ for which
$\{1,\dots,n\}$ can be partitioned into $k$ sum-free sets. The values
$S(1),\dots,S(5)=1,4,13,44,160$ are known, the last one determined by
Heule~\cite{heule}. For $k=6$ the record has been $S(6)\ge536$ since 2000, when
Fredricksen and Sweet~\cite{fs} found it by a search over symmetric sum-free
partitions; it has been quoted as the best known bound ever since. We improve it
by one.

\begin{theorem}\label{thm:537}
$S(6)\ge 537$.
\end{theorem}

\begin{proof}
The following six sets are pairwise disjoint, their union is $\{1,\dots,537\}$, and
none of them contains $x$, $y$ and $x+y$. We list every element, as in the
explicit constructions of Fredricksen and Sweet~\cite{fs}. These properties
are verified by checking all sums $x+y$ within each set, including $x=y$.
\par\smallskip
\noindent\begin{minipage}{\linewidth}
\textbf{Color 1 (76 integers).}
\begin{lstlisting}
  1   5  12  15  19  21  23  25  29  32  39  43  45  49  95
102 106 108 112 119 122 126 128 130 132 136 139 146 150 152
156 163 166 170 172 174 176 180 358 362 364 366 368 372 375
382 386 388 392 399 402 406 408 410 412 416 419 426 430 432
436 443 489 493 495 499 506 509 513 515 517 519 523 526 533
537
\end{lstlisting}
\end{minipage}\par\smallskip

\noindent\begin{minipage}{\linewidth}
\textbf{Color 2 (76 integers).}
\begin{lstlisting}
  2   6   7  17  18  22  27  37  41  42  46  51  66  70  81
 89  90 100 101 105 109 110 113 114 124 125 129 133 134 144
145 149 153 157 158 168 169 192 346 369 370 380 381 385 389
393 394 404 405 409 413 414 424 425 428 429 433 437 438 448
449 457 468 472 487 492 496 497 501 511 516 520 521 531 532
536
\end{lstlisting}
\end{minipage}\par\smallskip

\noindent\begin{minipage}{\linewidth}
\textbf{Color 3 (88 integers).}
\begin{lstlisting}
  3   8   9  14  20  24  26  30  31  36  47  53  59  63  65
 69  76  82  86  88  92  98 104 115 120 121 127 131 137 142
143 148 154 159 160 165 171 177 181 182 187 188 193 204 334
345 350 351 356 357 361 367 373 378 379 384 390 395 396 401
407 411 417 418 423 434 440 446 450 452 456 462 469 473 475
479 485 491 502 507 508 512 514 518 524 529 530 535
\end{lstlisting}
\end{minipage}\par\smallskip

\noindent\begin{minipage}{\linewidth}
\textbf{Color 4 (70 integers).}
\begin{lstlisting}
  4  10  11  13  16  28  33  34  35  40  55  57  79  93  94
 96  99 111 116 117 118 123 135 138 140 141 147 155 161 162
164 167 179 185 203 335 353 359 371 374 376 377 383 391 397
398 400 403 415 420 421 422 427 439 442 444 445 459 481 483
498 503 504 505 510 522 525 527 528 534
\end{lstlisting}
\end{minipage}\par\smallskip

\noindent\begin{minipage}{\linewidth}
\textbf{Color 5 (76 integers).}
\begin{lstlisting}
 38  44  48  50  52  54  56  58  60  61  62  64  67  68  71
 72  73  74  75  77  78  80  83  84  85  87  91  97 103 107
173 189 196 197 201 209 212 236 302 326 329 337 341 342 349
365 431 435 441 447 451 453 454 455 458 460 461 463 464 465
466 467 470 471 474 476 477 478 480 482 484 486 488 490 494
500
\end{lstlisting}
\end{minipage}\par\smallskip

\noindent\begin{minipage}{\linewidth}
\textbf{Color 6 (151 integers).}
\begin{lstlisting}
151 175 178 183 184 186 190 191 194 195 198 199 200 202 205
206 207 208 210 211 213 214 215 216 217 218 219 220 221 222
223 224 225 226 227 228 229 230 231 232 233 234 235 237 238
239 240 241 242 243 244 245 246 247 248 249 250 251 252 253
254 255 256 257 258 259 260 261 262 263 264 265 266 267 268
269 270 271 272 273 274 275 276 277 278 279 280 281 282 283
284 285 286 287 288 289 290 291 292 293 294 295 296 297 298
299 300 301 303 304 305 306 307 308 309 310 311 312 313 314
315 316 317 318 319 320 321 322 323 324 325 327 328 330 331
332 333 336 338 339 340 343 344 347 348 352 354 355 360 363
387
\end{lstlisting}
\end{minipage}\par\smallskip

\end{proof}

Theorem~\ref{thm:537} needs only this one partition. Section~\ref{sec:consequences}
lists what it implies for other Schur and Ramsey numbers. In Section~\ref{sec:families} we
describe all $131$ partitions of $\{1,\dots,537\}$ into six sum-free sets that our
search produced.
Section~\ref{sec:method} describes the construction of the families.
The complete partitions and supplementary tables are available at
\url{https://github.com/jeffchen006/Schur-Number-6}.

Throughout, $N=538$, $\fold(x)=\min(x,N-x)$, and a colouring $c$ of
$\{1,\dots,537\}$ is \emph{palindromic} if $c(x)=c(N-x)$ for all $x$. We number the
colours $1,\dots,6$ in the order of their least elements. We write
$C_j=c^{-1}(j)$ for the $j$-th color class.

\section{Consequences for other bounds}\label{sec:consequences}
Classical recurrences of Schur and Abbott--Hanson~\cite{schur,abbott},
and the template constructions of Rowley and Ageron et al.~\cite{rowley,ageron},
turn smaller partitions into larger ones. We use
\[
S(k+2)\ge10S(k)+2\quad\text{\cite{bengone}},\qquad
S(k+5)\ge380S(k)+148\quad\text{\cite{ageron}}.
\]
A sum-free $k$-partition of $[1,n]$ also gives an edge coloring of $K_{n+1}$
without monochromatic triangles: color $ij$ by the class of $|i-j|$.
Consequently, $R_k(3)\ge S(k)+2$.

\begin{corollary}\label{cor:consequences}
The lower bounds in the column headed ``This work'' in Table~\ref{tab:bounds} hold.
\end{corollary}

\begin{table}[tbp]
\caption{Lower bounds in the cited literature and the bounds obtained here.
The six- and eight-color bounds are supported by explicit partitions.}
\label{tab:bounds}
\centering
\begin{tabular}{lrrl}
\toprule
Quantity & Previous bound & This work & Previous source\\
\midrule
$S(6)$    & $536$          & $537$            & Fredricksen--Sweet~\cite{fs}\\
$S(8)$    & $5\,362$       & $5\,372$         & Bengone et al.~\cite{bengone}\\
$S(11)$   & $203\,828$     & $204\,208$       & Ageron et al.~\cite{ageron}\\
$S(13)$   & $2\,038\,282$ & $2\,042\,082$    & Bengone et al.~\cite{bengone}\\
$R_6(3)$  & $538$          & $539$            & Fredricksen--Sweet~\cite{fs,xxer}\\
\bottomrule
\end{tabular}
\end{table}

\begin{proof}
$S(8)\ge10\cdot537+2$, $S(11)\ge380\cdot537+148$,
$S(13)\ge10\cdot S(11)+2\ge10\cdot204\,208+2$, and $R_6(3)\ge S(6)+2$. For $S(8)$ we also
built the partition explicitly: apply the width-$10$ template of
\cite[Table~II]{bengone} to the partition of Theorem~\ref{thm:537}, with colours $7$
and $8$ for the labels $A$ and $B$. This gives a partition of
$\{1,\dots,5372\}$ into eight sum-free sets.
\end{proof}

Bounds for larger $k$ that are derived through these recurrences change in the same
way; we list only the four cases in which we checked the previously published
values.

\section{The 537-partitions and their families}\label{sec:families}
The example above is palindromic, and each of its classes is even sum-free modulo
$538$. The least elements of its classes are $1,2,3,4,38,151$, and the sixth class
contains the whole interval $[237,301]$. These features are captured by the
following family of colourings.

\begin{definition}\label{def:frameg}
A colouring $c\colon\{1,\dots,537\}\to\{1,\dots,6\}$ whose classes are sum-free
modulo $538$ is in \emph{frame $g$} if, with $Q=151$ and $d=38$,
\begin{align*}
 \fold(x)<j &\ \Rightarrow\ c(x)\neq j \quad (j=1,2,3,4),\\
 \fold(x)<d \ \text{or}\ Q-d<\fold(x)<Q &\ \Rightarrow\ c(x)\neq 5,\\
 \fold(x)\in\{d,\,N-2Q\} &\ \Rightarrow\ c(x)=5,\\
 \fold(x)<Q &\ \Rightarrow\ c(x)\neq 6,\\
 \fold(x)=Q \ \text{or}\ \fold(x)>N-2Q &\ \Rightarrow\ c(x)=6.
\end{align*}
\end{definition}

In words: the sixth colour does not occur below $151$ and fills the middle of the
interval, and the fifth colour does not occur below $38$ or just below $151$.

The first class of the example has a simple description as well. Put
\begin{align*}
A&=\{1,5,12,15,19,21,23,25,29,32,39,43,45,49\},\\
B&=A\cup(151-A)\cup\{151+a: a\in A,\ a\le 29\}\cup\{95\},
\end{align*}
a set of $38$ numbers, and $C^*=\{x\in\{1,\dots,537\}:\fold(x)\in B\}$, a set of $76$
numbers. The first class of the example is exactly $C^*$.

\begin{proposition}\label{prop:families}
The $131$ partitions of $\{1,\dots,537\}$ into six sum-free sets listed in
the supplementary material are pairwise distinct, every class of each of them is
sum-free modulo $538$, and no one of them is obtained from another by a map
$x\mapsto rx \bmod 538$ with $\gcd(r,538)=1$ followed by renaming the colours.
They form the five families summarized in Table~\ref{tab:families}.
The example of Theorem~\ref{thm:537} is $\mathcal F_1.001$.
Each partition in $\mathcal F_1$--$\mathcal F_4$ agrees with the example in at least
$462$ of the $537$ positions.
\end{proposition}

\begin{proof}
For each partition, check all sums within each class modulo $538$.
To check inequivalence, compare all images under multiplication by units modulo
$538$, with colors renumbered in the order of their least elements.
The family properties follow by counting the quantities in Table~\ref{tab:families}.
\end{proof}

For Table~\ref{tab:families}, $d_{\rm pal}$ counts pairs $\{x,538-x\}$
with different colors, $\Delta_1=|C_1\mathbin{\triangle}C^*|$, and
\[
h=\bigl|\{y\in\{237,\dots,269\}:\{y,538-y\}\not\subseteq C_6\}\bigr|.
\]
Thus $h$ counts each folded value once if either corresponding position is
outside color 6, including the single central position $y=269$.
Family $\mathcal F_5$ is measured after
$x\mapsto43x\bmod538$ and color renumbering; its original class minima are
$1,2,3,4,7,16$. Families $\mathcal F_1$--$\mathcal F_3$ satisfy frame $g$.
Family $\mathcal F_4$ satisfies the frame with the solid-middle condition removed.
These families describe the retained sample, not all possible 537-partitions.

\begin{table}[tbp]
\caption{The 131 retained partitions. Intervals give the observed minimum and
maximum within a family. Features of $\mathcal F_5$ use the normalization stated
in the text. The quantity $h$ counts folded values, not individual positions.
All entries are computed from the 131 explicit partitions.}
\label{tab:families}
\centering
\begin{tabular}{lrrrr}
\toprule
Family & Count & $d_{\rm pal}$ & $\Delta_1$ & $h$\\
\midrule
$\mathcal F_1$ & 107 & 0 & 0 & 0\\
$\mathcal F_2$ & 14  & 0 & 2 & 0\\
$\mathcal F_3$ & 5   & 2--12 & 0 & 0\\
$\mathcal F_4$ & 4   & 0 & 2--8 & 1--13\\
$\mathcal F_5$ & 1   & 5 & 3 & 5\\
\midrule
Total & 131 & & &\\
\bottomrule
\end{tabular}
\end{table}

\section{Construction of the families}\label{sec:method}
The palindromic families are obtained by imposing reflection symmetry and the
sum-free conditions modulo $538$. For $\mathcal F_1$ and $\mathcal F_2$, we also
impose frame $g$. The two families are distinguished by whether the first class
is $C^*$ or differs from it in one reflected pair. For $\mathcal F_3$, we retain
the frame conditions but remove reflection symmetry. For $\mathcal F_4$, we
retain reflection symmetry but remove the condition that color 6 contains every
position with $f(x)>236$.

For $\mathcal F_5$, let $c$ be the coloring in Theorem~\ref{thm:537} and put
\[
c_m(x)=c(mx\bmod538),\qquad m\in\{3,25,221\},
\]
with the colors of each $c_m$ renamed to maximize agreement with $c$.
We then require a new coloring $c'$ to satisfy
\[
c'(x)\in\{c_3(x),c_{25}(x),c_{221}(x)\},\qquad 1\le x\le537,
\]
together with sum-freeness modulo $538$. This construction does not impose
reflection symmetry or the solid-middle condition.

The example in Theorem~\ref{thm:537} is sufficient for the new lower bound.
The other partitions describe variations of its structure; they do not form an
exhaustive classification of six-color partitions of $[1,537]$.

\end{document}